\documentclass[11pt, oneside]{amsart}  	
\usepackage{geometry}               
\usepackage{graphicx}				
\usepackage{amsthm,amsmath,amssymb,graphicx,mathrsfs,braket,latexsym}
\usepackage[all]{xy}
								
\newtheorem{thm}{Theorem}[section]
\newtheorem{prop}[thm]{Proposition}
\newtheorem{lem}[thm]{Lemma}

\theoremstyle{definition}
\newtheorem{defi}[thm]{Definition}
\newtheorem{rem}[thm]{Remark}
\newtheorem{ex}[thm]{Example}

\newtheorem{nota}[thm]{Notation}

\DeclareMathOperator{\supp}{supp}

\DeclareMathOperator{\spa}{span}
\DeclareMathOperator{\Patch}{Patch}
\DeclareMathOperator{\Tiling}{Tiling}
\DeclareMathOperator{\cl}{Cl}
\DeclareMathOperator{\cpt}{Cpt}

\DeclareMathOperator{\Map}{Map}
\DeclareMathOperator{\UD}{UD}
\DeclareMathOperator{\Del}{Del}

\newcommand{\LDRd}{\overset{\Rd}{\rightarrow}}

\newcommand{\Zpo}{\mathbb{Z}_{>0}}

\newcommand{\calT}{\mathcal{T}}
\newcommand{\calU}{\mathcal{U}}
\newcommand{\calA}{\mathcal{A}}

\newcommand{\calP}{\mathcal{P}}
\newcommand{\calQ}{\mathcal{Q}}
\newcommand{\calL}{\mathcal{L}}
\newcommand{\e}{\varepsilon}
\newcommand{\Rd}{\mathbb{R}^d}

\newcommand{\LD}{\overset{\mathrm{LD}}{\rightsquigarrow}}
\newcommand{\MLD}{\overset{\mathrm{MLD}}{\leftrightsquigarrow}}

\newcommand{\rhoT}{\rho_{\mathbb{T}}}

\newcommand{\sci}{\wedge}

\title[A new relation]{A new relation between geometric and dynamical properties of objects such as
tilings and Delone sets}
\author{Yasushi Nagai}
\address{Montanuniversit\"at, Department Mathematik und Informationstechnologie,
Lehrstuhl f\"ur Mathematik und Statistik,
Franz Josef Strasse 18, A-8700 Leoben, Austria}
\email{yasushi.nagai@unileoben.ac.at}
\date{\today}							
\thanks{The author was supported by the project I3346 of the Japan Society for the Promotion of Science (JSPS) and the Austrian Science Fund (FWF)}

\begin{document}
\maketitle

\begin{abstract}
    Let $\calP$ be an object such as tiling, Delone set and weighted Dirac comb.
    There corresponds a dynamical system to $\calP$, called the corresponding dynamical
    system. Such dynamical systems are geometric analogues of symbolic dynamics.
     It is well-known that there are
    correspondences between geometric properties of $\calP$ and properties of the
    corresponding dynamical system. In this article we give a new  correspondence.
    In other words, we characterize the property that the group of topological
    eigenvalues for the corresponding dynamical system
     is not discrete, in terms of a geometric property of $\calP$.
\end{abstract}

\section{Introduction}
The discovery of quasicrystals showed that there are non-periodic mathematical
structure that have long-range order. Since the discovery, objects such as tilings
and Delone sets that are ``ordered'' have been studied intensively.
Here, a tiling is a cover of $\Rd$ by tiles such as polygons that overlap only on their
boundaries. A Delone set, sometimes called a separated net, is a subset $D$ of $\Rd$ such
that distances of two distinct points are bounded from below and
there are no arbitrary large balls that do not intersect $D$.

The term ``ordered'' has many interpretations.
The most important interpretation is that the objects are pure point diffractive.
This is a mathematical description of the physical diffraction pattern to have
only bright spots without diffuse background.

More vaguely, objects are ``ordered'' if local environments are correlated with
one another, even if they are widely separated.
Being pure point diffractive can be understood under this vague interpretation, since
by \cite{solomyak1998spectrum}, \cite{MR2084582}, and \cite{MR2135448}
that condition is implied or equivalent to several forms of almost periodicity.

The goal of this article is to relate a interpretation of ``order'' with a dynamical
property of objects such as tilings and Delone sets.
This is an extension of \cite{Nagai1}.

What is the dynamical property of those objects? The historical origin is in
symbolic dynamics. Let $w$ be a word, that is, a map $w\colon\mathbb{N}\rightarrow\calA$,
where $\calA$ is a finite set, and $\sigma\colon\calA^{\mathbb{N}}\rightarrow\calA^{\mathbb{N}}$ is
a left-shift. Define a compact set $X_{w}=\overline{\{\sigma^n(w)\mid n\in\mathbb{N}\}}$,
where we take the closure with respect to the product topology on $\calA^{\mathbb{N}}$.
The restriction of the map $\sigma$ on $X_{w}$ defines a topological dynamical system.

Similarly, consider an object $\calP$ such as tiling and Delone set in $\Rd$.
For each $x\in\Rd$, $\calP+x$ denotes the translation of $\calP$ by $x$.
We set $X_{\calP}=\overline{\{\calP+x\mid x\in\Rd\}}$,
where we take a closure with respect
to a topology similar to the product topology in $\calA^{\mathbb{N}}$.
The group
$\Rd$ acts on $X_{\calP}$ by translation, and we get a topological dynamical system
$(X_{\calP},\Rd)$. The dynamical properties are the properties of this dynamical system.
It is well-known that there are correspondences between geometric properties of $\calP$
(properties on how patterns are distributed in $\calP$) and dynamical properties
of $\calP$. 

In this article we first study a Delone set $D$ and prove Theorem \ref{thm_nonlinearity}
in Section \ref{section_result_for}.
In this theorem we relate a dynamical condition and a condition on ``order'',
which gives a new correspondence between geometric and dynamical properties.
The precise statement of this phenomenon of ``order'' can be found in
Theorem \ref{thm_nonlinearity},
but in a plain language, it means as follows.
Consider three positive real numbers $L_1,L_2$ and $R$ with the following
property: take $x\in\Rd$ arbitrarily and consider the pattern $\calP=D\cap B(x,R)$
around the point $x$ with
the spherical window of radius $R$; then relative to $x$, there is a ``forbidden area''
of appearance of $\calP$ in $D$, where the translates of $\calP$ never appears;
the forbidden area has a form of stripe, consisting of ``bands'' of width $2L_2$ with
intervals $L_1$.
We will prove that, if we assume a dynamical condition that the set of eigenvalues is
 not discrete in $\Rd$, we can always take such $L_1,L_2$ and $R$ in such a way that
 $L_1$ and $L_2$ are  close to
any two given positive real number with respect to any error $\e$.
The converse will also be proved under a mild assumption.
The existence of such `forbidden area'' means that,
although the Delone set $D$ may not be non-periodic, there emerges a periodic structure
of stripe, that is shown in Figure \ref{figure_nonlinearity} in page
\pageref{figure_nonlinearity}.
We cannot tell what exactly happens in the distance, but we can say non-existence of
patterns in forbidden area. In this sense the behaviors of two distant parts of $D$
are correlated and $D$ has ``order'.

After proving this theorem, in Section \ref{section_generalization}
we prove the same theorem for another objects such as
tilings (with or without labels) and Delone multi-sets.
We simply ``convert'' these objects into Delone sets and use
Theorem \ref{thm_nonlinearity}.
We use a framework of abstract pattern space to state a theorem that is one-level more
 abstract than
Theorem \ref{thm_nonlinearity}.
In Theorem \ref{theorem_stripe_structure} and Theorem \ref{thm_converse_str_str},
we prove the similar equivalence for ``abstract patterns'', which is
an abstract notion that includes objects such as tilings and Delone sets as examples.

We finish the article with an appendix on the framework of abstract pattern spaces.

In \cite{Nagai1}, the author proved a weaker version of Theorem \ref{theorem_stripe_structure} for tilings.
The result there is weaker than those in this article in the sense that
we cannot take $R_1$ and $R_2$ arbitrarily. In this article we omit such a restriction
and also prove the converse. Furthermore, we prove such a correspondence for general
abstract patterns.

\begin{nota}
     In this article $\|\cdot\|$ denotes the Euclidean norm. The standard inner product
     of $\Rd$ is denoted by $\langle\cdot,\cdot\rangle$.

      Let $x$ be a point in $\Rd$ and $r$ be a positive real number. We define a closed
 ball
      \begin{align*}
          B(x,r)=\{y\in\Rd\mid \|x-y\|\leqq r\}.
      \end{align*}
      If $x=0$, the ball $B(x,r)$ will be denoted by $B_r$.

\end{nota}

\section{A result for Delone sets}\label{section_result_for}

	\begin{defi}\label{def_metric_on_T}
		We endow a metric $\rhoT$ on $\mathbb{T}$ by identifying $\mathbb{T}$ with $\mathbb{R}/2\pi\mathbb{Z}$. In other words we set 
		\begin{align*}
			\rhoT(e^{2\pi i\theta}, e^{2\pi i \theta'})=\min_{n\in\mathbb{Z}}|\theta-\theta'+n|
		\end{align*}
		for any $\theta, \theta'\in\mathbb{R}$. This gives a well-defined metric on $\mathbb{T}$ that generates the standard topology of $\mathbb{T}$.
	\end{defi}

       We first recall the definition of Delone set and FLC.
      \begin{defi}
            Let $D$ be a subset of $\Rd$.
             \begin{enumerate}
	      \item For each $R>0$, we say $D$ is \emph{$R$-relatively dense} if, for any
		    $x\in\Rd$ there is $y\in D$ with $\|x-y\|<R$.
		    The set $D$ is said to be \emph{relatively dense}
		    if it is $R$-relatively
		    dense for some $R>0$.
	     \item For each $r>0$, we say $D$  is \emph{$r$-uniformly discrete}
		   if for each
		   $x,y\in D$, we have either $x=y$ or $\|x-y\|>r$.
		    The set $D$ is said to be \emph{uniformly discrete} if it is $r$-uniformly
		   discrete for some $r>0$.
	     \item For each $R,r>0$, we say $D$ is a \emph{$(R,r)$-Delone set} if $D$ is
		   $R$-relatively dense and $r$-uniformly discrete. The set $D$ is
		   called a \emph{Delone set}
		   if it is a $(R,r)$-Delone set for some $R>0$ and
		   $r>0$.
	     \end{enumerate}
      \end{defi}

      \begin{defi}
           Let $D$ be a Delone set in $\Rd$. We say $D$ has \emph{finite local complexity
           (FLC)} if
          for any compact $K\subset \Rd$, the set
          \begin{align*}
	        \{(D-x)\cap K\mid x\in\Rd\}
	  \end{align*}
           is finite modulo translation.
      \end{defi}

      For the space of all uniformly discrete sets one may define a uniform
      structure, the topology
      by which is called the \emph{local matching topology}
      (see for example \cite{Solomyak_dynamics}).
      The local matching topology will be defined in Definition \ref{def_local_mat_top}
      in a more abstract setting, but for uniformly discrete sets, the metric defined by
      \begin{align*}
            \rho(D_1,D_2)=\inf\Delta(D_1,D_2),
      \end{align*} 
      where
      \begin{align*}
            \Delta(D_1,D_2)=\{r\in (0,\frac{1}{\sqrt{2}})\mid
            \text{there are $x,y\in B_r$ such that}\\
             \text{$(D_1+x)\cap B_{1/r}=(D_2+y)\cap B_{1/r}$)}\},
      \end{align*}
      defines a local matching topology.
      
       \begin{defi}
	    Let $D$ be a Delone set. We define the continuous hull $X_D$ of $D$ via
	     \begin{align*}
	            X_D=\overline{\{D+x\mid x\in\Rd\}},
	     \end{align*}
	     where the closure is taken with respect to the local matching topology.
	     The dynamical system $(X_D,\Rd)$, which is obtained by the action of $\Rd$
	     on $X_D$ by translation, is called the \emph{corresponding dynamical system
	for $D$}.
       \end{defi}
       
      We then define subsets of $\Rd$ that is ``stripe-shaped''.
       \begin{defi}\label{def_S_abR1R2}
	    Take $a,b\in\Rd$ such that $\|a\|=1$.
	    Take also positive real numbers $L_1,L_2>0$.
	    Set
	     \begin{align*}
	          S(a,b,L_1,L_2)=\{x\in\Rd\mid \langle x-b,a\rangle\in
	              L_1\mathbb{Z}+[-L_2,L_2]\}.
	     \end{align*}
       \end{defi}

       \begin{rem}
	     $S(a,b,L_1,L_2)$ is the union of ``bands'' with width $2L_2$ and
	     intervals $L_1$, which is depicted in
	Figure \ref{figure_nonlinearity} in page
	     \pageref{figure_nonlinearity}.
       \end{rem}

       The following notion describes a property of the distribution of patterns
       in Delone sets.
       \begin{defi}\label{def_stripe_structure_for_tilings}
	    Let $D$ be a Delone set of $\Rd$ and $L_1,L_2>0$.
	   We say $D$ has $(L_1,L_2)$-stripe structure
	    if there are $a\in\Rd$ with $\|a\|=1$ and $R>0$ such that
	     \begin{align}
	          \{y\in\Rd\mid (D-x)\cap B(0,R)=(D-y)\cap B(0,R)\}
	           \subset S(a,x,L_1,L_2)\label{eq_def_stripe_structure}
	     \end{align}
	     for each $x\in\Rd$.
       \end{defi}

       \begin{rem}
	    In plain language, the inclusion \eqref{eq_def_stripe_structure}
	means that, if we take a subset $E\subset D$ around the point $x$
	which is large enough, there is a ``forbidden area'' of the appearance of
	the translate of $E$ in $D$.
	The forbidden area is a periodic one which is obtained by
	juxtaposing bands of width $2L_2$.
       (See Figure \ref{figure_nonlinearity} in page \pageref{figure_nonlinearity}.)
       \end{rem}

       Now we show the following theorem.

\begin{thm}\label{thm_nonlinearity}
       Take a Delone set $D$ which has FLC.
       Consider the following two conditions:
      \begin{enumerate}
       \item  $0\in\Rd$ is a limit point of the group of eigenvalues 
	             of the corresponding dynamical system $(X_D,\Rd)$.
	      In other words, the group of all topological eigenvalues is not
	      discrete in $\Rd$.
	      \item For any
       $R_1,R_2,\e>0$, there are $L_1,L_2>0$ such that
      \begin{enumerate}
       \item $|R_j-L_j|<\e$ for each $j=1,2$, and
	     \item $D$ has $(L_1,L_2)$-stripe structure.
      \end{enumerate} 
      \end{enumerate}
       Then the first condition always implies the second and if $D$ is repetitive
      the second one implies the first.
\end{thm}

			\begin{figure}[h]
		\begin{center}
		\includegraphics[width=1.0\columnwidth]{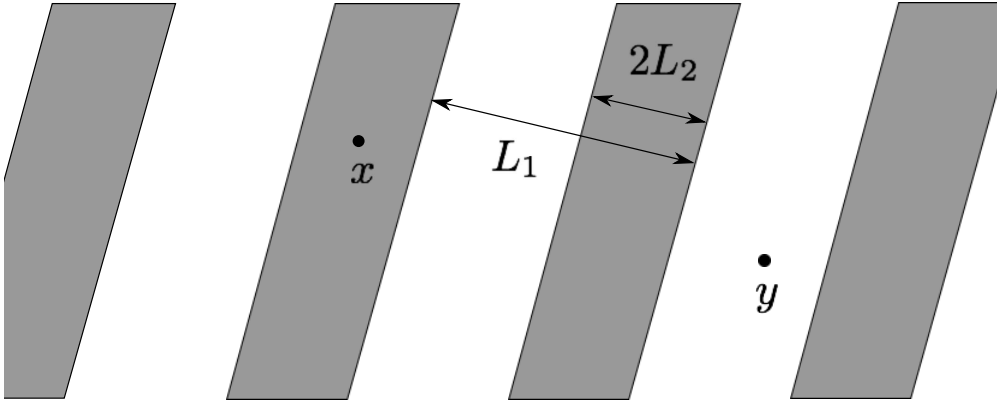}%
		\caption{$(L_1,L_2)$-stripe structure. The situation of the tiling
		 $\mathcal{T}$ around the point $x$ is different from the one around the
		 points, such as $y$, outside the shaded region.}
		\label{figure_nonlinearity}
		\end{center}
		\end{figure}

\begin{rem}
        Even if we know $\calT$ has stripe structure,
       we do not know how large the $R>0$ in Definition
        \ref{def_stripe_structure_for_tilings} is.
      If the set $D$ is a Meyer set, then
       by the characterization of Meyer sets
      \cite{MR1460032}, we see
       \begin{align*}
	    \{\chi\in(\mathbb{R}^d)^{\hat{}}\mid \text{for any $x\in D$, we have
	     $|\chi(x)-1|<\e$}\}
       \end{align*}
       is relatively dense for any $\e>0$.
        Thus we have a stripe-shaped forbidden area of the appearance of points in $D$
         (in other words, we can take $R=0$).
         However, the width of each bands may be small.
        In Theorem \ref{thm_nonlinearity}, we can choose the width and interval of bands
        arbitrarily up to small error.
\end{rem}	

           We now prove that the first condition in Theorem \ref{thm_nonlinearity}
	   implies the second. That the second implies the first under the assumption
	   of repetitivity is proved later. We first prove the following lemmas.
         \begin{lem}\label{lem1_span_of_small_eigenvalues}
	     Let $G$ be a (not necessarily closed) subgroup of $\Rd$.
	     Set
	    \begin{align}
	         V=\bigcap_{r>0}\overline{\spa_{\mathbb{Z}}G\cap B(0,r)}.
	     \label{eq_def_of_V}
	    \end{align}
	  Then $V$ is a vector subspace of $\Rd$.
	 \end{lem}
	 \begin{proof}
	     First, $0\in V$.

	    Second, if $x,y\in V$, then $x+y\in V$.
	   Indeed, for any $r>0$ and $\e>0$ there are
	   $x',y'\in\spa_{\mathbb{Z}}G\cap B(0,r)$ such that
	  $\|x-x'\|<\e$ and $\|y-y'\|<\e$.
	  Since $x'+y'\in\spa_{\mathbb{Z}}G\cap B(0,r)$ and
	  $\|x+y-(x'+y')\|<2\e$, we see $x+y\in\overline{\spa_{\mathbb{Z}}G\cap B(0,r)}$.

	  Third, we show that if $x\in V$ and $n\in\Zpo$, then
	  $\frac{1}{n}x\in V$.
	  For any $m\in\Zpo$, let $\mathcal{B}$ be a maximal linear independent subset
	  of $G\cap B(0,1/m)$. We may take $\lambda_b\in\mathbb{R}$ for each
	  $b\in\mathcal{B}$ such that $x=\sum_{b\in\mathcal{B}}\lambda_{b}b$,
	  since $x\in\overline{\spa_{\mathbb{Z}}G\cap B(0,1/m)}\subset
	  \spa_{\mathbb{R}}\mathcal{B}$.
	  We may take $l_b\in\mathbb{Z}$ for each $b\in\mathcal{B}$ such that
	  $|\lambda_b-nl_b|<n$ for each $b$. Set $x_m=\sum_{b\in\mathcal{B}}l_bb$.
	  We have
	  \begin{align*}
	      \|x-nx_m\|=&
	      \|\sum_{b\in\mathcal{B}}\lambda_bb-\sum_{b\in\mathcal{B}}nl_bb\|\\
	           \leqq&\sum|\lambda_b-nl_b|\|b\|\\
	           \leqq&\frac{nd}{m},
	  \end{align*}
	  and so
	  \begin{align*}
	   \|\frac{1}{n}x-x_m\|\leqq\frac{d}{m}.
	  \end{align*}
	  For any $r>0$ and $\varepsilon>0$, if $m$ is large enough,
	  $x_m\in\spa_{\mathbb{Z}}G\cap B(0,r)$ and $\|\frac{1}{n}x-x_m\|<\varepsilon$.
	  This shows $\frac{1}{n}x\in V$.

	  Finally, by the second and the third part of this proof, if $\lambda$ is an
	  rational number, then $\lambda x\in V$. Since $V$ is closed,
	  this holds even if $\lambda$ is irrational.
	 \end{proof}

\begin{lem}\label{lem2_span_of_small_eigenvalues}
       Let $G$ be a subgroup of $\Rd$ and define $V$ by \eqref{eq_def_of_V}.
        If $0\in\Rd$ is a limit point of $G$, then the dimension of $V$ is more than $0$.
\end{lem}
\begin{proof}
       For each integer $n>0$ there is $x_n\in G$ such that $0<\|x_n\|<1/n$.
       We may find $k_n\in\mathbb{Z}$ such that $1/2<\|k_nx_n\|<3/2$.
      The sequence $(k_nx_n)$ admits a limit point $x$. Then $x\neq 0$. Moreover,
      $x\in \overline{\spa_{\mathbb{Z}}G\cap B(0,r)}$ for each $r>0$ since
    $k_nx_n\in\spa_{\mathbb{Z}}G\cap B(0,r)$ for large $n$, and so
      $x\in V$.
\end{proof}
	 
       \begin{lem}[\cite{kellendonk2013meyer}, Lemma 4.1]\label{lemma_Kellendonk_weak_equiv}
	     Let $D$ be an FLC
	Delone set in $\Rd$ and $\chi$ a continuous character of $\Rd$.
	     Then $\chi$ is an eigenvalue for the topological dynamical system
	    $(X_D,\Rd)$ if and only if $\chi$ is a weakly $D$-equivariant function, that
	is, for any $\e>0$ there is $R>0$ such that
	\begin{align*}
	      \rho_{\mathbb{T}}(\chi(x),\chi(y))<\e
	\end{align*}
	for any $x,y\in\Rd$ with
	\begin{align*}
	     (D-x)\cap B(0,R)=(D-y)\cap B(0,R).
	\end{align*}
       \end{lem}

	 We now prove one way of implications in Theorem \ref{thm_nonlinearity}.
        \begin{prop}\label{prop_Delone_stripe_str}
	    Let $D$ be a Delone set of $\Rd$ which has FLC.
	     Suppose that $0$ is a limit point of the set of all topological eigenvalues
	    for $(X_D,\Rd)$.
	    Then for any $L_1,L_2>0$ and $\e>0$, there are $R_1,R_2>0$ such that
	     \begin{enumerate}
	      \item $|R_j-L_j|<\e$ for each $j=1,2$, and
	      \item $D$ has $(R_1,R_2)$-stripe structure.
	     \end{enumerate}
	\end{prop}
        \begin{proof}
              By Lemma \ref{lem1_span_of_small_eigenvalues} and Lemma
	 \ref{lem2_span_of_small_eigenvalues}, we can take an eigenvalue
	      $a$ such that $|\frac{1}{\|a\|}-L_1|<\e$. Take $r>0$ such that
	      $\frac{r}{\|a\|}=L_2$. We set $R_1=\frac{1}{\|a\|}$ and
	      $R_2=\frac{r}{\|a\|}$.

	      Since the character $\chi_a$ is weakly $D$-equivariant, there is
	       $R>0$ such that $x,y\in\Rd$ and
	       \begin{align}
		   (D-x)\cap B(0,R)=(D-y)\cap B(0,R)
		\label{eq1_many_eigen_stripe_structure}
	       \end{align}
	       imply
	      \begin{align*}
	           \rho_{\mathbb{T}}(\chi_a(x),\chi_a(y))\leqq r.
	      \end{align*}
	      We will show that this $R$ satisfies the condition in
	      Definition \ref{def_stripe-structure}.
	      Take $x\in\Rd$ and fix it. If $y\in\Rd$ and
	 (\ref{eq1_many_eigen_stripe_structure}) holds,
	      then we have
	       \begin{align*}
		    |\langle y-x,a\rangle-n|\leqq r.
	       \end{align*}
	        for some $n\in\mathbb{Z}$.
	       We obtain
	       \begin{align*}
		\langle y-x,a\rangle\in\mathbb{Z}+[-r,r],
	       \end{align*}
	       and so $y\in S(\frac{1}{\|a\|}a,x,R_1,R_2)$.
	      We have proved
	      \begin{align*}
	           \{y\in\Rd\mid (D-y)\cap B(0,R)=(D-x)\cap B(0,R)\}
	      \end{align*}
	      is contained in $S(\frac{1}{\|a\|}a,x,R_1,R_2)$, and so
	       $D$ has $(R_1,R_2)$-stripe structure.
	\end{proof}

	In what follows we prove the remaining part
	of Theorem \ref{thm_nonlinearity}
under the assumption of repetitivity.

\begin{lem}\label{lem1_for_converse_stripe_str}
     Let $D$ be a repetitive $(R,r)-$Delone set in $\Rd$.
     Take $x_0\in\Rd$ and $R_0>R$ arbitrarily.
     Set
      \begin{align*}
            E=\{x\in\Rd\mid (D-x_0)\cap B(0,R_0)=(D-x)\cap B(0,R_0)\}.
      \end{align*}
      Then $E$ is a Delone set and $D\LDRd E$.
\end{lem}
\begin{proof}
      Take $x$ and $y$ from $E$.
      Since $R_0$ is greater than $R$, the set
      \begin{align*}
           (D-x)\cap B(0,R_0)=(D-y)\cap B(0,R_0)
      \end{align*}
      is not empty. Take $z$ from this set.
      We see $x+z,y+z\in D$ and $\|x-y\|=\|x+z-(y+z)\|$ is either $0$ or
      greater than $r$. This shows $E$ is uniformly discrete with respect to $r$.

      Next, since $D$ is repetitive, there is $R>0$ such that
       for any $x\in\Rd$ there is $y\in\Rd$ with
       \begin{enumerate}
	\item $\|x-(x_0+y)\|<R$, and
	      \item $D\cap B(x_0+y,R_0)=(D\cap B(x_0,R_0))+y$.
       \end{enumerate}
       Let $x\in\Rd$ be an arbitrary element and $y\in\Rd$ satisfy the above two
       conditions. Then
      \begin{align*}
            (D-(x_0+y))\cap B(0,R_0)=(D-x_0)\cap B(0,R_0),
      \end{align*}
      and so $x_0+y\in E$. We have shown that $E\cap B(x,R)\neq\emptyset$ and
      $E$ is relatively dense with respect to $R$.

      Finally, we show that $D\LDRd E$.
      Take $x,y\in\Rd$ and $L>0$ arbitrarily and assume
      \begin{align*}
            (D-x)\cap B(0,R_0+L)=(D-y)\cap B(0,R_0+L).
      \end{align*}
      To prove
      \begin{align}
            (E-x)\cap B(0,L)=(E-y)\cap B(0,L)\label{eq_goal_D_LDRd_E},
      \end{align}
       we take $z\in E$ such that $z-x\in B(0,L)$. Then
       \begin{align*}
	    (D-x_0)\cap B(0,R_0)=&(D-z)\cap B(0.R_0)\\
	                       =&((D-x)\cap B(0,R_0+L)\cap B(z-x,R_0))+x-z\\
	                       =&((D-y)\cap B(0,R_0+L)\cap B(z-x,R_0))+x-z\\
	                       =&(D+x-y-z)\cap B(0,R_0).
       \end{align*}
        This implies that $z+y-x\in E$ and so $z-x\in E-y$.
        We have shown
        \begin{align*}
	  (E-x)\cap B(0,L)\subset (E-y)\cap B(0,L),
	\end{align*}
        and since the proof for the reverse inclusion is the same, we have
       \eqref{eq_goal_D_LDRd_E}.
\end{proof}

The following lemma is essentially \cite[Proposition 4.3]{kellendonk2013meyer}, where
the authors use the theory of pattern-equivariant cohomology. We give a proof without
using this theory.

\begin{lem}\label{lem2_stripe_str_converse}
       Let $D$ be an $(R,r)$-Delone set in $\Rd$ and assume $D$ is repetitive.
       Let $f\colon\Rd\rightarrow\mathbb{R}$ be a (not necessarily continuous)
        bounded function such that $a,b,c,d\in D$ and $a-b=c-d$ imply
       $f(a)-f(b)=f(c)-f(d)$. Then for any $\e>0$ there is a Delone $D_{\e}$ in $\Rd$
      such that
      \begin{enumerate}
        \item $D\supset D_{\e}$,
       \item $D\LDRd D_{\e}$, and
	    \item if $a,b\in D_{\e}$, then $|f(a)-f(b)|<\e$.
      \end{enumerate}
\end{lem}
\begin{proof}
    We may replace $f$ with $f+C$ for some constant $C\in\mathbb{R}$ so that we may
   assume
    \begin{align*}
         M=\sup_{a\in D} f(a)=-\inf_{a\in D} f(a).
    \end{align*}
    For any $\e>0$ there are $a_0$ and $b_0$ in $D$ such that
    $f(a_0)>M-\e/2$ and $f(b_0)<-M+\e/2$.
    Take $c_0\in D$ and fix it. If $R_0>R$ is sufficiently large, we have
    $a_0,b_0\in B(c_0,R_0)$. Set
   \begin{align*}
        D_{\e}=\{x\in\Rd\mid (D-c_0)\cap B(0,R_0)=(D-x)\cap B(0,R_0)\}.
   \end{align*}
   Then $D_{\e}\subset D$ and by Lemma \ref{lem1_for_converse_stripe_str},
   $D_{\e}$ is Delone and $D\LDRd D_{\e}$.

   Next, take $a\in D_{\e}$ arbitrarily and we show $|f(a)-f(c_0)|<\e/2$.
   Since
   \begin{align*}
          (D\cap B(c_0,R_0))-c_0=(D-a)\cap B(0,R_0),
   \end{align*}
    by definition of $R_0$ we see $a_0-c_0+a\in D$ and $b_0-c_0+a\in D$.
    Since
    \begin{align*}
           f(a_0-c_0+a)-f(b_0-c_0+a)=f(a_0)-f(b_0)>2M-\e,
    \end{align*}
     we have either $f(a_0-c_0+a)>M-\e/2$ or $f(b_0-c_0+a)<-M+\e/2$. In the former
     case, we see
    \begin{align*}
          f(a)-f(c_0)=&f(a+a_0-c_0)-f(a_0)+f(a)-f(a+a_0-c_0)-f(c_0)+f(a_0)\\
                     =&f(a+a_0-c_0)-f(a_0)\\
                     \in& (-\e/2,\e/2),
    \end{align*}
    since
     \begin{align*}
          f(a)-f(a+a_0-c_0)=f(c_0)-f(a_0)
     \end{align*}
     by the assumption on $f$.
     Similarly in the latter case
     \begin{align*}
           f(a)-f(c_0)\in (\e/2,\e/2).
     \end{align*}

    Finally, if $a,b\in D_{\e}$, then by the previous paragraph
     \begin{align*}
           |f(a)-f(b)|\leqq |f(a)-f(c_0)|+|f(b)-f(c_0)|<\e,
     \end{align*}
     which completes the proof.
\end{proof}

The following lemma can be found in \cite[p.123]{kellendonk2013meyer}, but we do not
omit the proof which does not use the theory of cohomology, for the reader's convenience.
\begin{lem}\label{lem2_converse_stripe_str}
      Let $D$ be a repetitive Delone set in $\Rd$.
      Let $a_0$ and $b_0$ be elements of $\Rd$. Assume if $x\in D$ we have
      \begin{align*}
              |\langle x-b_0,a_0\rangle- n|<1/4
      \end{align*}
       for some $n\in\mathbb{Z}$. Then the character
       $\chi_{a_0}\colon\Rd\ni x\mapsto e^{2\pi i\langle x,a_0\rangle}\in\mathbb{T}$,
       is weakly $D$-equivariant.
\end{lem}
\begin{proof}
      We may take $\theta\colon\Rd\rightarrow [-\pi,\pi)$ such that for any
     $x\in\Rd$
       \begin{align*}
	    2\pi\langle x-b_0,a_0\rangle=\theta(x)+2n\pi
       \end{align*}
       for some $n\in\mathbb{Z}$. For any $a,b,c,d\in D$ such that $a-b=c-d$, we have
       \begin{align*}
	    e^{i(\theta(a)-\theta(b))}=&e^{2\pi i(\langle a-b_0,a_0\rangle-
	                                      2\pi\langle b-b_0,a_0\rangle)}\\
	                            =&e^{2\pi i\langle a-b,a_0\rangle}\\
	                            =&e^{2\pi i\langle c-d,a_0\rangle}\\
	                            =&e^{i\theta(c)-\theta(d)}.
       \end{align*}
       Since $\theta(a),\theta(b),\theta(c)$ and $\theta(d)$ are in $(-\pi/2,\pi/2)$
       by the assumption, we see $\theta(a)-\theta(b)=\theta(c)-\theta(d)$.
       By Lemma \ref{lem2_stripe_str_converse}, for each $\e>0$ there is a Delone
       $D_{\e}\subset D$ such that $D\LDRd D_{\e}$ and $|\theta(a)-\theta(b)|<\e$ for
       any $a,b\in D_{\e}$.
       Let $R_1$ be a constant for the local derivability $D\LDRd D_{\e}$ and
        $R_2>0$ be such that $D_{\e}$ is relatively dense with respect to $R_2$.
        If $x,y\in\Rd$ and
       \begin{align}
	      (D-x)\cap B(0,R_1+R_2)=(D-y)\cap B(0,R_1+R_2)
	\label{eq_lem3_converse_stripe_str}
       \end{align}
        then
       \begin{align*}
	      (D_{\e}-x)\cap B(0,R_2)=(D_{\e}-y)\cap B(0,R_2).
       \end{align*}
        Take $z\in D_{\e}$ such that $z-x\in B(0,R_2)$.
        Then $z-x+y\in D_{\e}$ and
        $|\theta(z)-\theta(z-x+y)|<\e$.
        Using
         \begin{align*}
	         |e^{2\pi i\langle x,a_0\rangle}-e^{2\pi i\langle y, a_0\rangle}|
	          =&|e^{2\pi i\langle z-b_0, a_0\rangle}-e^{2\pi i\langle z-x+y-b_0,a_0\rangle}|\\
	          =&|e^{i\theta(z)}-e^{i\theta (z-x+y)}|,
	 \end{align*}
         we see that for any $\eta>0$, if $\e>0$ is small enough,
          the equation \eqref{eq_lem3_converse_stripe_str} implies that
          \begin{align*}
	      \rho_{\mathbb{T}}(\chi_{a_0}(x),\chi_{a_0}(y))<\eta.
	  \end{align*}
\end{proof}

\begin{prop}\label{lem4_converse_stripe_str}
       Let $D$ be an FLC and repetitive Delone set in $\Rd$.
       Suppose for any $R_1,R_2>0$ and $\e>0$, there are $L_1,L_2>0$ such that
       \begin{enumerate}
	\item $|L_j-R_j|<\e$ for each $j=1,2$, and
	      \item $D$ has $(L_1,L_2)$-stripe structure.
       \end{enumerate}
       Then $0$ is a limit point of the group of all topological eigenvalues
       for $(X_D,\Rd)$.
\end{prop}
\begin{proof}
       For any $R_1,R_2,\e>0$ we take $L_1$ and $L_2$ as in the assumption.
       By the definition of stripe structure (Definition \ref{def_stripe-structure}),
      there are $a_0\in\Rd$ with $\|a_0\|=1$ and $R>0$ such that
      \begin{align*}
           E=\{y\in\Rd\mid (D-y)\cap B(0,R)=(D-x)\cap B(0,R)\}\subset S(a,x,L_1,L_2)
      \end{align*}
       for each $x\in\Rd$. Since we can take arbitrarily large $R$,
       by Lemma \ref{lem1_for_converse_stripe_str} $E$ is Delone and $D\LDRd E$.

       If $R_1>4R_2$ and $\e$ is small enough, then $L_1>4L_2$.
        Then if $y\in E$, we have
       \begin{align*}
	    \langle y-x,a_0\rangle\in L_1\mathbb{Z}+[-L_2,L_2],
       \end{align*}
       and so
       \begin{align*}
	     \langle y-x,\frac{1}{L_1}a_0\rangle \in \mathbb{Z}+(-1/4,1/4).
       \end{align*}
        The set $E$ is repetitive and has FLC since these properties are inherited under
       local derivability.
      Using Lemma \ref{lem2_converse_stripe_str}, we see $\chi_{(1/L_1)a_0}$ is
        weakly $E$-equivariant. By Lemma \ref{lemma_Kellendonk_weak_equiv}, we see
       $\chi_{(1/L_1)a_0}$ is a topological eigenvalue for $(X_E,\Rd)$.
       Since $(X_E,\Rd)$ is a factor of $(X_D,\Rd)$ by the fact that
      $D\LDRd E$,
     we see $\chi_{(1/L_0)a_0}$ is a topological eigenvalue for $(X_D,\Rd)$.
      Since $L_1$ may be arbitrarily large, we see $0$ is a limit point of the set
     of topological eigenvalues for $(X_D,\Rd)$.
\end{proof}

This finishes the proof of Theorem \ref{thm_nonlinearity}.

\subsection{Stripe structure for model sets}
In this subsection we show model sets satisfies the equivalent conditions in Theorem
\ref{thm_nonlinearity}.

First let us recall the definition of cut-and-project scheme and model set.
Let $G$ and $H$ be locally compact abelian groups and $\calL$ be a cocompact lattice of
$G\times H$. The triple $(G,H,\calL)$ is called a \emph{cut-and-project scheme} if the
following conditions are satisfied:
\begin{enumerate}
 \item the restriction $\pi_1|_{\calL}$ of the projection
       $\pi_1\colon G\times H\rightarrow G$ is injective;
 \item the image $\pi_2(\calL)$ of $\calL$ by the projection
       $\pi_2\colon G\times H\rightarrow H$ is dense in $H$.
\end{enumerate}

We consider a compact subset $W$ of $H$ with the following
conditions:
\begin{enumerate}
 \item  $\overline{W^{\circ}}=W$;
 \item $\{h\in H\mid W+h=W\}=\{0\}$;
 \item any Haar measure of the boundary $\partial W$ of $W$ is zero;
\item $\partial W\cap \pi_2(\calL)=\emptyset$.
\end{enumerate}
For a cut-and-project scheme $(G,H,\calL)$ and a compact subset $W\subset H$
with the above conditions, we define
a \emph{regular and generic model set} $\Lambda(G,H,\calL,W)$ via
\begin{align*}
     \Lambda(G,H,\calL,W)=\{\pi_1(l)\mid l\in\calL,\pi_2(l)\in W\}.
\end{align*}
We omit ``regular and generic'' below and just call them model sets.

Given a cut-and-project scheme $(G,H,\calL)$, the torus $\mathfrak{T}$ is by definition
$\mathfrak{T}=(G\times H)/\calL$. Denote the quotient map
$G\times H\rightarrow\mathfrak{T}$ by $q$. The group $G$ acts on $\mathfrak{T}$ via
\begin{align*}
      \mathfrak{T}\times G\ni(q(x,h),y)\mapsto q(x+y,h)\in\mathfrak{T}.
\end{align*}
Let $0_H$ be the identity element of $H$ and define a map $\iota_G$ via
\begin{align*}
 G\ni x\mapsto (x,0_H)\in G\times H.
\end{align*}

We construct eigenfunctions of the dynamical system $(X_{\Lambda},G)$ by using the
following torus parametrization.
\begin{thm}[\cite{MR1798991}]
       Let $\Lambda$ be a model set.
      Then there is a factor map $\beta\colon\Omega_{\Lambda}\rightarrow\mathfrak{T}$
      that sends $\Lambda$ to $q(0_G,0_H)$, where $0_G$ is the identity element of $G$.
      This factor is a maximal equicontinuous factor.
\end{thm}

\begin{lem}
      Let $\chi$ be a continuous character of $\mathfrak{T}$. Then $\chi\circ\beta$ is
      a continuous eigenfunction of the dynamical system $(X_{\Lambda},G)$ with
       the eigencharacter $(q\circ\iota_G)\hat{{}}(\chi)$.
       Conversely, if $f\colon X_{\Lambda}\rightarrow\mathbb{T}$ is a continuous
      eigenfunction such that $f(\Lambda)=1$,
     then there is a continuous character $\chi\in\hat{\mathfrak{T}}$
      such that $f=\chi\circ\beta$.
      Consequently, the group of all topological eigencharacters for $(X_{\Lambda},G)$
      coincides with $(q\circ\iota_G)\hat{{}}(\hat{\mathfrak{T}})$.
\end{lem}
\begin{proof}
     This follows from the factor map $\beta$ is a maximal equicontinuous factor.
\end{proof}

Given a cut-and-project scheme $(G,H,\calL)$, we have a commutative diagram
\[\xymatrix{
 G  \ar[r]^-{\iota_G}  \ar[rd]& G\times H  \ar[d]^q
 & H\ar[l]_-{\iota_H} \ar[ld]\\
                      &(G\times H)/\calL=\mathfrak{T}     &
		      }\]
where the map $\iota_G$ is as above and $\iota_H$ sends $s\in H$ to
$(0_G,s)$. 
By taking Pontryagin dual, we have a commutative diagram

\[\xymatrix{
\hat{G}      & \hat{G}\times\hat{H} \ar[l]_{\hat{\iota_G}} \ar[r]^{\hat{\iota_H}} &\hat{H}\\
            &   \calL^{\circledast} \ar[lu]\ar[u]^{\iota} \ar[ru]  &       \\
	    &  \hat{\mathfrak{T}}\ar[luu]\ar[u]^-{\rotatebox{90}{$\sim$}}\ar[ruu] &
} 
\]
Here, $\hat{\iota_G}$ and $\hat{\iota_H}$ are duals of $\iota_G$ and $\iota_H$, which are
projections. The map $\iota$ is the inclusion.
The symbol $\calL^{\circledast}$ denotes the dual lattice:
\begin{align*}
    \calL^{\circledast}=\{(\chi_G,\chi_H)\in\hat{G}\times\hat{H}\mid
 \text{$\langle s,\chi_G\rangle\langle t,\chi_H\rangle=1$ for all $(s,t)\in\calL$}\}.
\end{align*}
Since $q\circ\iota_G$ has a dense image, its dual is injective. Since $q\circ\iota_H$ is
injective, the dual has a dense image. We have a cut-and-project scheme
$(\hat{G},\hat{H},\calL^{\circledast})$, called \emph{the dual cut-and-project scheme}.
For each $\chi\in\hat{\iota_G}(\calL^{\circledast})$, there is a unique $\chi^*\in\hat{H}$
such that $(\chi,\chi^*)\in\calL^{\circledast}$.

In order to prove many model sets satisfy the conditions of the Theorem
\ref{thm_nonlinearity}, we prove the following.

\begin{lem}\label{lem_discrete_eigenvalues_Hhatcompact}
      Assume $G=\Rd$.
     If the set of all continuous eigencharacters
     $(q\circ\iota_G)\hat{{}}(\hat{\mathfrak{T}})=\hat{\iota_G}(\calL^{\circledast})$
    is discrete, then
     $\hat{H}$ is compact.
\end{lem}
\begin{proof}
     If the assumption  holds, there is a linearly independent set
      $\{b_1,b_2,\ldots b_k\}$ of $\Rd$ such that
    \begin{align*}
          \hat{\iota_G}(\calL^{\circledast})=\spa_{\mathbb{Z}}\{b_1,b_2,\ldots ,b_k\}.
    \end{align*}
    Let $V$ be the orthogonal space in $\Rd$ of
     $\spa_{\mathbb{R}}\{b_1,b_2,\ldots ,b_k\}$.
    Define $U$ by
    \begin{align*}
         U=\{\sum_j\lambda_jb_j+x\mid \lambda_j\in (-1,1),x\in V\}.
    \end{align*}
    Then $U$ is open, $U\cap\hat{\iota_G}(\calL^{\circledast})=\{0\}$, and
      $U+\hat{\iota_G}(\calL^{\circledast})=\Rd$.

     Take a relatively compact, open and non-empty subset $O$ of $\hat{H}$. Since
      the quotient map
 $\hat{G}\times\hat{H}\rightarrow (\hat{G}\times\hat{H})/\calL^{\circledast}$,
      also denoted by $q$,
      is open and surjective, the set
      \begin{align*}
            \{q((U+\chi_G)\times (O+\chi_H))\mid
       \chi_G\in\hat{\iota_G}(\calL^{\circledast}),\chi_H\in\hat{H}\}
      \end{align*}
      is an open cover of $(\hat{G}\times\hat{H})/\calL^{\circledast}$.
      Since $(\hat{G}\times\hat{H})/\calL^{\circledast}$
    is compact, we can take a finite set
       $F_G\subset\hat{\iota_G}(\calL^{\circledast})$ and
      $F_H\subset \hat{H}$ such that
       \begin{align*}
            \{q((U+\chi_G)\times (O+\chi_H))\mid
       \chi_G\in F_G,\chi_H\in F_H\}
      \end{align*}
      is a finite subcover.

      If $\xi_H\in\hat{H}$, then there are $\chi_G\in F_G$, $\eta_H\in F_H$,
      $\zeta_G\in U$ and $\zeta_H\in O$ such that
       \begin{align*}
	    q(0,\xi_H)&=q(\zeta_G+\chi_G,\zeta_H+\eta_H)\\
	             &=q(\zeta_G,\zeta_H+\eta_H-\chi_G^*),
       \end{align*}
        where $\chi_G^*\in\hat{H}$ is the unique element satisfying
        $(\chi_G,\chi_G^*)\in\calL^{\circledast}$.
        Since $(\zeta_G,\zeta_H+\eta_H-\chi_G^*-\xi_H)\in\calL^{\circledast}$,
     by definition of
      $U$ we have $\zeta_G=0$ and $\zeta_H+\eta_H-\chi_G^{*}-\xi_H=0$.
      We see $\xi_H\in O+\eta_H-\chi_G^{*}$.

      By this observation we have
      \begin{align*}
         \hat{H}=\bigcup_{\eta_H\in F_H,\chi_G\in F_G}O+\eta_H-\chi_G^*.
      \end{align*}
      Since $O$ is relatively compact, $\hat{H}$ is compact.
\end{proof}

\begin{prop}
      Let $\Lambda$ be a (regular and generic) model set of $\Rd$
    constructed from a cut-and-project scheme
     $(\Rd,H,\calL)$ where $H$ is not discrete. Then $\Lambda$ satisfies the
      equivalent conditions in Theorem \ref{thm_nonlinearity}.
\end{prop}
\begin{proof}
       Since $H$ is not discrete, $\hat{H}$ is not compact.
       Hence the set of all eigenvalues is not discrete by Lemma \ref{lem_discrete_eigenvalues_Hhatcompact}.
\end{proof}

\section{A generalization}\label{section_generalization}
In this section we generalize Theorem \ref{thm_nonlinearity} by replacing the Delone set
$D$ with ``abstract patterns''. The theory of abstract pattern spaces and
 abstract patterns is introduced in
\cite{Nagai3rd} and its outline is found in Appendix.
The reader may simply replace
``abstract pattern'' below with ``Delone multi sets'', ``tilings'' (with or without
labels), ``weighted Dirac combs'' and ``abstract pattern spaces'' with
each class of these objects.
These objects have a natural cutting-off operation
$\sci$. If $\calP$ is an abstract pattern and $C$ is a closed
subset of $\Rd$, the abstract
pattern $\calP\sci C$ is obtained by simply forgetting the behavior of $\calP$ outside
$C$. For example, if $\calP$ is a tiling, $\calP\sci C$ is the set of all tiles in $\calP$
that are included in $C$.

Below we deal with an abstract pattern that is MLD with a Delone set in $\Rd$. The terms
``locally derivable'' and ``mutually locally derivable (MLD)''
are defined in Appendix. Note that
a sufficient condition for an abstract pattern to admit such a Delone set is given
in \cite{Nagai3rd}.

\label{stripe_structure}

       
       \begin{defi}\label{def_stripe-structure}
	 Take two positive real numbers $R_1,R_2$.
         Let $\Pi$ be an abstract pattern space over $(\Rd,\Rd)$.
	 An abstract pattern $\calP\in\Pi$ is said to admit $(R_1,R_2)$-stripe
	 structure if there are $a\in\Rd$ with $\|a\|=1$ and
	$R>0$ such that, for any $x\in\Rd$, the set
	\begin{align*}
	   \{y\in\Rd\mid (\calP-x)\sci B(0,R)=(\calP-y)\sci B(0,R)\}
	\end{align*}
	  is contained in $S(a,x,R_1,R_2)$ (Definition \ref{def_S_abR1R2}).
        \end{defi}

        \begin{lem}\label{lem_local_derivable_strpie_structure}
	     Let $\Pi_1,\Pi_2$ be abstract pattern spaces over $(\Rd,\Rd)$,
	     $R_1$ and $R_2$ positive real numbers.
	     Take $\calP_1\in\Pi_1$ and $\calP_2\in\Pi_2$ and assume
	     $\calP_2\LD\calP_1$
	     (that is, $\calP_1$ is locally derivable from $\calP_2$).
	      If $\calP_1$ has $(R_1,R_2)$-stripe structure,
	       then $\calP_2$ has $(R_2,R_2)$-stripe structure.
	\end{lem}
        \begin{proof}
	      There is $R>0$ and $a$ as in Definition \ref{def_stripe-structure}.
	      In other words, 
	      $x,y\in\Rd$ and
	      \begin{align}
	         (\calP_1-x)\sci B(0,R)=(\calP_1-y)\sci B(0,R)
	         \label{eq_local_derivable_stripe_structure}
	      \end{align}
	      imply $y\in S(a,x,R_1,R_2)$.

	       Since $\calP_2\overset{\Rd}{\rightarrow}\calP_1$, we can take a constant $R_0>0$ as
	       in Definition \ref{def_local_derive} with respect to
	       $x_0=y_0=0$.
	       If $x,y\in\Rd$ and
	       \begin{align*}
		     (\calP_2-x)\sci B(0,R+R_0)=(\calP_2-y)\sci B(0,R+R_0),
	       \end{align*}
	       then (\ref{eq_local_derivable_stripe_structure}) holds,
	        and so $y\in S(a,x,R_1,R_2)$.
	        We have proved $\calP_2$ has $(R_1,R_2)$-stripe structure
	         with respect to $R+R_0$.
	\end{proof}

	We generalize Theorem \ref{thm_nonlinearity}.
	First we generalize one direction of the theorem, by using the fact that
	if an abstract pattern $\calP$ is MLD with a Delone set $D$, then the
	two dynamical systems $(X_{\calP},\Rd)$ and $(X_D,\Rd)$ are topologically
	conjugate (Lemma \ref{compact_inherit}).
	
        \begin{thm}\label{theorem_stripe_structure}
	      Let $\Pi$ be an abstract pattern space over $(\Rd,\Rd)$ on which
	     the local matching uniform structure is complete.
	      Take an abstract pattern $\calP\in\Pi$, and assume it
	     has FLC and there is a Delone set $D$ in $\Rd$ with $\calP\MLD D$.
	     For example, take an FLC tiling of $\Rd$ of finite tile type.
	      Suppose that $0\in\Rd$ is a limit point of the set of topological
	     eigenvalues of the corresponding dynamical system $(X_{\calP},\Rd)$.
	     Then for any $R_1,R_2,\e>0$, there are $L_1,L_2>0$ such that
	      \begin{enumerate}
	       \item $|R_j-L_j|<\e$ for each $j=1,2$, and
		     \item $\calP$ has $(L_1,L_2)$-stripe structure.
	      \end{enumerate}
	\end{thm}
	\begin{proof}
	 The set of eigenvalues of the dynamical system $(X_D,\Rd)$
	       is the same as the one of $(X_{\calP},\Rd)$;
	       $D$ has $(L_1,L_2)$-stripe structure, for some $L_1,L_2>0$ with
	       $|R_j-L_j|<\e$
	      by Proposition \ref{prop_Delone_stripe_str};
	      by Lemma \ref{lem_local_derivable_strpie_structure}, $\calP$ also has
	       $(L_1,L_2)$-stripe structure.
	\end{proof}


\begin{rem}\label{rem_after_nonlinearity}
        Consider a primitive FLC  substitution rule (inflation rule)
        with injective substitution map and
       such that the expansion map $\phi$ is diagonalizable and all the eigenvalues are
       algebraic conjugates of the same multiplicity.
        If the spectrum of the expansion map is a Pisot family, then for the self-affine
       tilings $\calT$ for this substitution,
       $0\in\Rd$ is a limit point of the set of eigenvalues, and so the assumption of
       Theorem \ref{theorem_stripe_structure} holds, by the following 
        argument.
        The tiling $\calT$ is non-periodic and so by \cite{So},
        $a\in\Rd$ is an eigenvalue if
        $\lim_n e^{2\pi i\langle a,\phi^n(x)\rangle}=1$ for any return vector
        $x$. There is a non-zero eigenvalue $a$ by \cite{MR2851885}, and
         $(\phi^*)^{-k}(a)$, where * denotes the adjoint,
         is an eigenvalue for all $k>0$; since $\lim_k(\phi^*)^{-k}(a)=0$, we have
         arbitrary small non-zero eigenvalues.
\end{rem}

We then prove the converse of Theorem \ref{theorem_stripe_structure} under the assumption
of repetitivity.

\begin{thm}\label{thm_converse_str_str}
      Let $\Pi$ be an abstract pattern space over $(\Rd,\Rd)$ on which
      the local matching uniform structure is complete, and
      $\calP$ an element of $\Pi$ that is repetitive and
      such that there exists a Delone set $D$ in $\Rd$
       with $\calP\MLD D$.
     Suppose for any $R_1,R_2>0$ and $\e>0$ there are $L_1,L_2>0$ such that
     \begin{enumerate}
      \item $|R_j-L_j|<\e$ for each $j=1,2$, and
	    \item $\calP$ has $(L_1,L_2)$-stripe structure.
     \end{enumerate}
     Then $0$ is a limit point of the set of all topological eigenvalues for
      $(X_{\calP},\Rd)$.
\end{thm}
\begin{proof}
    By Lemma \ref{lem_local_derivable_strpie_structure} and
       Proposition
 \ref{lem4_converse_stripe_str}, we see $0$ is a limit point of the set of
        topological eigenvalues for $(X_D,\Rd)$. Since $(X_{\calP},\Rd)$ and
     $(X_D,\Rd)$ are topologically conjugate and $D$ is repetitive
     (Lemma \ref{lem_repeti_inherits}), we obtain the conclusion. 
\end{proof}

\appendix
\section{The framework of abstract pattern spaces}
\subsection{The definition and examples of abstract pattern space}

Several objects of interest in aperiodic order,
such as tilings and Delone sets, admits an operation of ``cutting off''.
This cutting-off operation satisfies three axioms. We use these axioms to capture
the objects such as tilings and Delone sets in a unified manner.
A set with such a cutting-off operation is called an abstract pattern space as defined below
and elements of abstract pattern spaces are called abstract patterns.
Objects of interest, such as tilings and Delone sets, are abstract patterns.

\begin{defi}
      The set of all closed sets of $\Rd$ is denoted by $\cl(\Rd)$.
\end{defi}

\begin{defi}
     A non-empty set $\Pi$ is called an abstract pattern space over $\Rd$ if it
     is equipped with an operation
      \begin{align}
           \Pi\times\cl(\Rd)\ni(\calP,C)\mapsto \calP\sci C\in\Pi
       \label{eq_scioors_operation}
      \end{align}
      such that the following conditions are satisfied:
       \begin{enumerate}
	\item for each $\calP\in\Pi$ and $C_1,C_2\in\cl(\Rd)$, we have
	      \begin{align*}
	           \calP\sci (C_1\cap C_2)=(\calP\sci C_1)\sci C_2,
	      \end{align*}
	       and
	\item for any $\calP\in\Pi$ there is $\supp\calP\in\cl(\Rd)$ such that
	      \begin{align*}
	            \calP\sci C=\calP\iff C\supset \supp\calP
	      \end{align*}
	      for $C\in\cl(\Rd)$.
       \end{enumerate}
        The closed set $\supp\calP$ in the second condition is unique and is
       called the support
       of $\calP$.
       The operation in \eqref{eq_scioors_operation} is called the cutting-off operation
        of the abstract pattern space $\Pi$.
        Elements of $\Pi$ are called abstract patterns.
      If there is a group action
      $\Rd\curvearrowright\Pi$, which sends a pair $(\calP,x)$ of $\calP\in\Pi$ and
 $x\in\Rd$ to $\calP+x$, such that
	      \begin{align*}
	            (\calP\sci C)+x=(\calP+x)\sci (C+x),
	      \end{align*}
  	 for each $\calP\in\Pi$, $C\in\cl(\Rd)$ and $x\in\Rd$,
        then $\Pi$ is called an abstract pattern space over $(\Rd,\Rd)$.
         (The first $\Rd$ in $(\Rd,\Rd)$ is a space and the second is a group. We may
       replace the first $\Rd$ with another space and the second with another group that
     acts on the space continuously.
       Abstract pattern spaces over $(\Rd,\Rd)$ are abstract pattern spaces over $\Rd$.)
       A nonempty subset $\Sigma$ of $\Pi$ that is closed under the group action
      is called a subshift of $\Pi$.
\end{defi}

We can also consider an abstract pattern space over $(\Rd,\Gamma)$, where $\Gamma$ is a group
of Euclidean motions, but we do not deal with them in this article.

We give several examples of abstract pattern spaces and subshifts.

\begin{ex}\label{exapmle_patches_without_labels}
     A non-empty, open and bounded subset of $\Rd$ is called a tile in $\Rd$.
     A collection $\calP$ of tiles in $X$ such that $S,T\in\calP$ and
    $S\cap T\neq\emptyset$ imply $S=T$, is called a patch.
      For a patch $\calP$, its support is defined by
     $\supp\calP=\overline{\bigcup_{T\in\calP}T}$.
     A patch $\calP$ is called a tiling if $\supp\calP=\Rd$.

     The set $\Patch(\Rd)$ of all patches of $\Rd$ is an abstract pattern space over $(\Rd,\Rd)$
     by a cutting-off operation
     \begin{align*}
           \calP\sci C=\{T\in\calP\mid T\subset C\}
     \end{align*}
     and a group action
      \begin{align*}
           \calP+x=\{T+x\mid T\in\calP\}.
      \end{align*}
       (The support defined just above satisfies the axiom for abstract pattern spaces.)
      The set $\Tiling(\Rd)$ of all tilings in $\Rd$ is a subshift of $\Patch(\Rd)$.
\end{ex}

In the above examples tiles are defined as open sets, since by this we do not need to
give labels (just give punctures to tiles which play the role of labels).
Below we also treat compact sets with labels as tiles, which is more standard.
The difference of definitions is not essential and we do not care about it.
The essential structure is cutting-off operation and group action.

\begin{ex}\label{example_patch_with_labels}
      Let $L$ be a finite set.
      A pair $(T,l)$ of a non-empty
   compact subset in $\Rd$ and an element $l\in L$ is called a
      $L$-labeled tile.
      A collection $\calP$ of $L$-labeled tiles is called a $L$-labeled patch if
       $(S,l),(T,l')\in\calP$ and $S^{\circ}\cap T^{\circ}\neq\emptyset$ imply
       $S=T$ and $l=l'$.
       For a $L$-labeled patch $\calP$, its support is defined by
        $\supp\calP=\overline{\bigcup_{(T,l)\in\calP}T}$.
         $L$-labeled patches $\calP$ with $\supp\calP=\Rd$
 are called a $L$-labeled tilings.
        The set $\Patch_L(\Rd)$ of all $L$-labeled patches  is an
      abstract pattern space over
         $(\Rd,\Rd)$ by a cutting-off operation
         \begin{align*}
	      \calP\sci C=\{(T,l)\in\calP\mid T\subset C\},
	 \end{align*}
         and a group action
          \begin{align*}
	       \calP+x=\{(T+x,l)\mid (T,l)\in\calP\}.
	  \end{align*}
          The set $\Tiling_L(\Rd)$ of all $L$-labeled tilings is a subshift of
           $\Patch_L(\Rd)$.
\end{ex}

\begin{ex}\label{example_maps}
      Let $Y$ be a nonempty set and $y_0\in Y$.
      The abstract pattern space $\Map(X,Y,y_0)$ over $(\Rd,\Rd)$ is defined as
      follows: as a set, it is equal to the set $\Map(X,Y)$ of all maps from $X$
      to $Y$; the cutting-off operation is defined by
      \begin{align*}
            (f\sci C)(x)=
            \begin{cases}
	          f(x)&\text{ if $x\in C$}\\
                  y_0&\text{ if $x\notin C$;}
	    \end{cases}
      \end{align*}
      the action $\Rd\curvearrowright\Map(X,Y)$ is defined by
      \begin{align*}
           (f+t)(x)=f(x-t).
      \end{align*}

      If $Y$ has a topology, the space $C(X,Y)$ of all continuous maps from $X$ to $Y$
      is a subshift of $\Map(X,Y,y_0)$.
\end{ex}

\begin{ex}\label{exapmle_closed_discrete_subset}
      On $2^{\Rd}$, the set of all subsets of $\Rd$, we may regard the usual intersection
     \begin{align*}
           2^{\Rd}\times \cl(\Rd)\ni(D,C)\mapsto D\cap C\in 2^{\Rd}
     \end{align*}
      as a cutting-off operation and $2^{\Rd}$ is an abstract pattern space over $\Rd$.
       With the usual group action $\Rd\curvearrowright 2^{\Rd}$ by translation,
      $2^{\Rd}$ is an abstract pattern space over $(\Rd,\Rd)$.
      $\cl(\Rd)$ is a ``abstract pattern subspace'' of $2^{\Rd}$, which means it is
     an abstract pattern space by restricting the cutting-off operation.
            Let $\mathfrak{D}_{\Rd}$ be the set of all discrete and closed subsets of
     $\Rd$. $\mathfrak{D}_{\Rd}$ is an abstract pattern subspace of $2^{\Rd}$.
      For example, the union of all Ammann bars for a Penrose tiling is an interesting
     element of $\cl(\Rd)$.

       For each $r>0$ let $\UD_r(\Rd)$ be the set of
      all $r$-uniformly discrete set of $\Rd$.
      $\UD_r(\Rd)$ is an abstract pattern subspace of $\mathfrak{D}_{\Rd}$.
      The set $\UD(\Rd)=\bigcup_{r>0}\UD_r(\Rd)$ of all uniformly discrete sets
      is also an abstract pattern subspace of
       $\mathfrak{D}_{\Rd}$.
        For each positive real numbers $R$ and $r$, the space
        $\Del_{R,r}(\Rd)$ of all $(R,r)$-Delone sets is a subshift of
         $\UD_r(\Rd)$.
\end{ex}

\begin{ex}\label{ex_uniformly_discrete_multi_sets}
      Let $L$ be a finite set.
      A subset $D$ of $\Rd\times L$ is called $r$-uniformly discrete $L$-set if
      $(x,l),(x',l')\in D$ and $\|x-x'\|\leqq r$ imply
	     $x=x'$ and $l=l'$.
      Define a cutting-off operation by
       \begin{align*}
	    D\sci C=\{(x,l)\in D\mid x\in C\},
       \end{align*}
        and a group action by
       \begin{align*}
	     D+t=\{(x+t,l)\mid (x,l)\in D\}.
       \end{align*}
       The set $\UD_r(\Rd,L)$ of all $r$-uniformly discrete $L$-sets is an
       abstract pattern space
       over $(\Rd,\Rd)$.
      $D\in\UD_r(X,L)$ is called a $(R,r)$-Delone $L$-set, where $R>0$,
      if for any $x\in \Rd$ there is $(y,l)\in D$ such that $\|x-y\|<R$.
      The set $\Del_{R,r}(X,L)$ of all $(R,r)$-Delone $L$-set is a subshift of
      $\UD_{r}(X,L)$.
\end{ex}

\begin{ex}\label{example_measures}
       Let $C_c(\Rd)$ be the set of all continuous, complex-valued and compact-support
     functions over $\Rd$. Endow this space the inductive limit topology, then the
      dual space $C_c(\Rd)'$ as a topological vector space is a pattern space over
      $(\Rd,\Rd)$, as follows: for each element $\Phi\in C_c(\Rd)'$ there exist a
      positive measure $m$ on $\Rd$ and a Borel $\mathbb{T}$-valued function $u$ such
      that
    \begin{align*}
         \Phi(f)=\int fudm
    \end{align*}
     for each $f\in C_c(\Rd)$; the cutting-off operation is defined via
    \begin{align*}
        (\Phi\sci C)(f)=\int_Cfudm,
    \end{align*}
     where $f\in C_c(\Rd)$ is arbitrary; the group action
     $\Rd\curvearrowright C_c(\Rd)'$ is defined via
      \begin{align*}
            (\Phi+t)(f)=\Phi(f-t),
      \end{align*}
      where $(f-t)(x)=f(x+t)$.
      Weighted Dirac combs are interesting examples of abstract patterns in this
      abstract pattern space. The set of all translation-bounded measures is a
 subshift of $C_c(\Rd)'$.
\end{ex}

\subsection{Local derivability}
\label{subsection_local_derivable}

Local derivability was defined in \cite{MR1132337} for tilings 
in $\Rd$.
Here we generalize it and define local derivability
 for  two abstract patterns $\calP_1$ and $\calP_2$.

\begin{lem}\label{lem_local_derivability}
        Let $\Pi_1$ and $\Pi_2$ be abstract pattern spaces over $(\Rd,\Rd)$.
         For two abstract patterns $\calP_1\in\Pi_1$ and $\calP_2\in\Pi_2$,
         the following two conditions are equivalent:
         \begin{enumerate}
	  \item There exist $x_{0}\in \Rd$, $y_0\in \Rd$ and $R_0\geqq 0$ such that 
		if $x,y\in\Rd$, $R\geqq 0$ and
		\begin{align*}
		     (\calP_1-x)\sci B(x_0,R+R_0)=(\calP_1-y)\sci B(x_0,R+R_0),
		\end{align*}
		then
		\begin{align*}
		     (\calP_2-x)\sci B(y_0,R)=(\calP_2-y)\sci B(y_0,R).
		\end{align*}
	  \item For any $x_{1}\in \Rd$ and  $y_1\in \Rd$
		there exists $R_1\geqq 0$ such that 
		if $x,y\in\Rd$, $R\geqq 0$ and
		\begin{align*}
		     (\calP_1-x)\sci B(x_1,R+R_1)=(\calP_1-y)\sci B(x_1,R+R_1),
		\end{align*}
		then
		\begin{align*}
		     (\calP_2-x)\sci B(y_1,R)=(\calP_2-y)\sci B(y_1,R).
		\end{align*}
	 \end{enumerate}
\end{lem}

\begin{defi}\label{def_local_derive}
        Let $\Pi_1$ and $\Pi_2$ be abstract pattern spaces over $(\Rd,\Rd)$.
        If $\calP_1\in\Pi_1$ and $\calP_2\in\Pi_2$ satisfy
        the two equivalent conditions in
        Lemma \ref{lem_local_derivability}, then we say $\calP_2$ is locally derivable from $\calP_1$
        and write $\calP_1\LD\calP_2$.       
        If both $\calP_1\LD\calP_2$ and $\calP_2\LD\calP_1$ hold,
        we say $\calP_1$ and $\calP_2$ are mutually locally derivable (MLD) and write
        $\calP_1\MLD\calP_2$. 
\end{defi}

\begin{rem}
     Often $\calP_1\LD\calP_2$ is equivalent to the condition that, for any compact
     $K\subset\Rd$ there is a compact $L\subset\Rd$ such that, if $x,y\in\Rd$ and
    \begin{align*}
         (\calP_1-x)\sci L=(\calP_1-y)\sci L,
    \end{align*}
     then
     \begin{align*}
          (\calP_2-x)\sci K=(\calP_2-y)\sci K.
     \end{align*}
\end{rem}

It is easy to show that $\LD$ is reflexive and transitive.

\begin{rem}
       Frequently an abstract pattern
      $\mathcal{P}$
     admits a Delone set $D$ in $\Rd$ such that $\calP\MLD D$. For example, if
     $\calP$ is either
      \begin{enumerate}
       \item a tiling with finitely many tiles up to translation, with or without label
	     for each tile, or
	\item a Delone multi set, or
	\item a weighted Dirac comb supported on a Delone set
      \end{enumerate}
       then there is such a Delone set for $\calP$.
\end{rem}

\subsection{The local matching topologies}
In \cite{Nagai4th} the author defined the local matching topology and uniform structure
for general abstract pattern spaces. Here we give an outline.
For the theory of uniform structures, see \cite{MR1726779}.

Below $\cpt(\Rd)$ denotes the set of all compact subsets of $\Rd$ and $\mathscr{V}$
denotes the set of all compact neighborhoods of $0\in\Rd$.

\begin{defi}
        For $K\in\cpt(\Rd)$ and $V\in\mathscr{V}$, set
        \begin{align*}
	      \calU_{K,V}=\{(\calP,\calQ)\in\Pi\times\Pi\mid 
                 \text{there is $x\in V$ such that 
                $\calP\sci K=(\calQ+x)\sci K$}\}.
	\end{align*}
\end{defi}

\begin{lem}\label{lem_UKV_fundamental_system_entourage}
       The set 
       \begin{align}
	     \{\calU_{K,V}\mid K\in\cpt(X), V\in\mathscr{V}\}
             \label{fundamental_system_entourage}
       \end{align}
        satisfies the axiom of fundamental system of entourages.
\end{lem}

\begin{defi}\label{def_local_mat_top}
      Let $\mathfrak{U}$ be the set of all entourages generated by
      (\ref{fundamental_system_entourage}).
      The uniform structure defined by $\mathfrak{U}$ is called the local matching
      uniform structure and the topology defined by it is called the local matching
     topology.
\end{defi}

Note that on several important sets $\Sigma$
of abstract patterns, the local
matching topology is metrizable (and so Hausdorff) and the local matching uniform
structure is complete. We list below several examples of such $\Sigma$'s:
\begin{prop}
       Let $\Sigma$ be either of the following sets of abstract patterns.
       \begin{enumerate}
	\item The set of all patches in $\Rd$
	      (Example \ref{exapmle_patches_without_labels}).
	 \item The set of all $L$-labeled patches
	       (Example \ref{example_patch_with_labels}).
	  \item The set of all uniformly discrete sets in $\Rd$
		(Example \ref{exapmle_closed_discrete_subset}).
	  \item The set of all uniformly discrete multi sets in $\Rd$
		(Example \ref{ex_uniformly_discrete_multi_sets}).
        \item   The set of all weighted Dirac combs in $\Rd$.
		(Example \ref{example_measures}).
       \end{enumerate}
       Then on $\Sigma$, the local matching topology is metrizable and the
      local matching uniform structure is complete.
\end{prop}

The continuous hulls, the corresponding dynamical systems, FLC and repetitivity are
important in aperiodic order. We define these concepts via the framework of
abstract pattern space.

\begin{defi}\label{def_conti_hull}
      For an abstract pattern $\calP$,
      define \emph{the continuous hull} $X_{\calP}$ of $\calP$ by
       \begin{align*}
	    X_{\calP}=\overline{\{\calP+x\mid x\in\Rd\}}
       \end{align*}
        where the closure is taken with respect to the local matching topology.
       The translation action of $\Rd$ on $X_{\calP}$ define a dynamical system
       $(X_{\calP},\Rd)$, which is called \emph{the corresponding dynamical system
      for $\calP$}.
\end{defi}

\begin{defi}
       An abstract pattern $\calP$ has \emph{finite local complexity (FLC)} if
      for each compact $K\subset\Rd$, the set
       \begin{align*}
	      \{(\calP+t)\sci K\mid t\in\Rd\}
       \end{align*}
        is finite modulo $\Rd$-action.
\end{defi}

\begin{lem}\label{compact_inherit}
      \begin{enumerate}
       \item Suppose on a set $\Sigma$ of abstract patterns, the local matching uniform
	     structure
	     is complete. Then if $\calP\in\Sigma$ has FLC, the continuous hull
	     $X_{\calP}$ is compact. 
      \item Let $\calP$ be an abstract pattern.
	    Suppose $\Sigma$ is a set of abstract patterns on which the local matching
	    uniform structure is complete, $\calQ\in\Sigma$ and $\calP\LD\calQ$.
	    Then the map
	    \begin{align*}
	          \{\calP+t\mid t\in\Rd\}\ni\calP+t\mapsto\calQ+t\in
	          \{\calQ+t\mid t\in\Rd\}
	    \end{align*}
	    extends to a factor map $X_{\calP}\rightarrow X_{\calQ}$.
	    In particular, if $X_{\calP}$ is compact, 
	    then $X_{\calQ}$ is compact.
      \end{enumerate}
\end{lem}

\begin{rem}
       Often if the continuous hull $X_{\calP}$ is compact, then $\calP$ has FLC.
       Thus in many cases, FLC is inherited by local derivability.
\end{rem}

We now define repetitivity.
\begin{defi}
        An abstract pattern $\calP$ is said to be \emph{repetitive} if, whenever we take
        compact $K\subset\Rd$, the set
        \begin{align*}
	     \{x\in\Rd\mid (\calP-x)\sci K=\calP\sci K\}
	\end{align*}
        is relatively dense.
\end{defi}

\begin{lem}\label{lem_repeti_inherits}
     Let $\calP_1$ and $\calP_2$ be abstract patterns and suppose $\calP_1\LD\calP_2$.
     If $\calP_1$ is repetitive then so is $\calP_2$.
\end{lem}

\bibliographystyle{amsplain}
\bibliography{tiling}

\end{document}